\documentclass{amsart}
\usepackage{graphicx}
\usepackage{amsmath}%in Winedt
\usepackage{amssymb}
\usepackage{amsthm}
\usepackage{esint}
\usepackage{mathrsfs}

\newtheorem{theorem}{Theorem}[section]
\newtheorem{lem}[theorem]{Lemma}

\theoremstyle{Corollary}
\newtheorem{cor}[theorem]{Corollary}
\newtheorem{prop}[theorem]{Proposition}
\newtheorem{defn}[theorem]{Definition}

\numberwithin{equation}{section}

\begin{document}

\title{Constructing $\lambda$-Angenent curve
by flow method}

\author{Pak Tung Ho}
\address{Department of Mathematics, Tamkang University Tamsui, New Taipei City 251301, Taiwan}

\email{paktungho@yahoo.com.hk}

%    General info
\subjclass[2000]{Primary 53C44; Secondary 53C22, 53C42}

\date{24th May, 2025.}

\keywords{Mean curvature flow; self-shrinker; Angenent torus
\newline The author was supported  by the National Science and Technology Council (NSTC),
Taiwan, with grant Number: 112-2115-M-032 -006 -MY2.}

\begin{abstract}

Using a modified curve shortening flow,
we construct $\lambda$-Angenent curve, 
which was first constructed 
by the shooting method. 
\end{abstract}

\maketitle

\section{Introduction}

The mean curvature flow is the gradient flow of the area functional
$$\int_\Sigma d\mathcal{H}^n(x).$$ 
More precisely, a smooth 
one-parameter family of hypersurfaces $\Sigma(t)\subset\mathbb{R}^{n+1}$, 
where $t\in [0,T)$, is a solution of the \textit{mean curvature flow}
if 
$$\frac{\partial x}{\partial t}
=\vec{H},$$
where $\vec{H}$ is the mean curvature vector of 
the hypersurface at point $x$. 
The mean curvature flow 
has been studied extensively for the last few decades; 
and it seems to be impossible to include all related references. 
We refer 
the readers to 
\cite{Colding}
for a survey of the mean curvature flow. 

A smooth hypersurface $\Sigma$ in $\mathbb{R}^{n=1}$ is called a
\textit{self-shrinker} if
it satisfies the equation 
$$\vec{H}+x^{\perp}/2=0.$$ 
Huisken
\cite{Huisken}
observed that if $\{\sqrt{-t}\Sigma\}_{t<0}$ 
is a solution of the mean curvature flow, 
then $\Sigma$ is a self-shrinker. 
Moreover, Huisken observed that a shrinker in $\mathbb{R}^{n+1}$
is the critical point of Gaussian area functional $$\mathcal{F}(\Sigma)=\int_\Sigma e^{-|x|^2/4}d\mathcal{H}^n(x).$$

Self-shrinkers are central objects of study in the
mean curvature flow, since they serve as singularity models \cite{Huisken1}.
Different techniques have been used to construct explicit
examples of self-shrinkers. See \cite{BNS,CLW,KKM,KM,N,R}. 
In particular, 
Angenent \cite{Angenent3} constructed
a rotationally symmetric
self-shrinking mean curvature flow
in $\mathbb{R}^{n+1}$
that is topologically $\mathbb{S}^1\times S^{n-1}$ before it becomes singular.
In fact, Angenent constructed a closed embedded hypersurface $\Sigma_n\subset \mathbb{R}^{n+1}$
that is rotationally symmetric such that $\{\sqrt{-t}\Sigma_n\}_{t<0}$ is 
a solution of the mean curvature flow. 
In view of the Huisken's observation mentioned above,  Angenent reduced the existence of a rotational symmetric shrinker
to the existence of a closed embedded geodesic in the half-plane $\mathbb{R}^2_+=\{(r,x): r>0\}$ equipped
with the metric
$$g=r^{2(n-1)}e^{-\frac{r^2+x^2}{2}} g_E$$
where
and $g_E=dr^2+dx^2$ is the Euclidean metric on $\mathbb{R}^2$.
Angenent used a shooting method to construct such a closed embedded geodesic.
See \cite{Drugan&Lee&Nguyen} for a survey of the shooting method.

It was observed by Chen and Sun in \cite[Appendix B]{Chen&Sun}
that if one replaces $n-1$ by some $\lambda>0$,
all the proofs in \cite{Angenent3} still work.
That is to say, using Angenent's shooting method,
one can construct a closed embedded geodesic
in the half-plane $\mathbb{R}^2_+=\{(r,x): r>0\}$ equipped
with the metric
\begin{equation}\label{metric}
g=\alpha^2(dr^2+dx^2)=\alpha^2 g_E
\end{equation}
where
$\alpha^2=r^{2(\lambda-1)}e^{-\frac{r^2+x^2}{2}}$ and $\lambda>1$ is a constant,
and $g_E$ is the Euclidean metric on $\mathbb{R}^2$.
Such a geodesic is called $\lambda$-Angenent curve in \cite{Chen&Sun},
and the $\lambda$-Angenent curve is used as the barrier. See \cite{Chen&Sun} for the details.

At the end of the paper \cite{Chen&Sun}, Chen and Sun said that
``It is worth noting that Drugan and Nguyen \cite{Drugan&Nguyen}
also constructed doughnuts using a variational method.
It is still an open question whether
their construction coincides with Angenent's construction.
It seems that their method can also be applied to construct $\lambda$-Angenent curves."
The following Theorem \ref{main} shows that this is indeed the case.

\begin{theorem}\label{main}
There exists a simple closed geodesic $\gamma_\infty(u)=(r(u),x(u))$, $u\in\mathbb{S}^1$,
in the half-plane $(\mathbb{R}^2_+,g)$.
Moreover, its length $L_g(\gamma_\infty)$ with respect to the metric $g$
is less than the length of the double cover of the half-line $x=0$.
\end{theorem}

Note that Drugan-Nguyen  \cite{Drugan&Nguyen} 
proved Theorem \ref{main} for the case when $\lambda=n-1$.
In fact, the proof of Theorem \ref{main} closely follows that of \cite{Drugan&Nguyen}.
We also remark that, similar to the case when $\lambda=n-1$,  
we do not know whether
the $\lambda$-Angenent curve constructed in Theorem \ref{main}
coincides with the one  constructed by the shooting method.

\section{Preliminary}

In this section, we collect some basic facts which will be used 
in the proof of our main theorem. 

The vectors $\displaystyle\frac{\partial}{\partial r}$ and
$\displaystyle\frac{\partial}{\partial x}$ form an orthonormal basis in $\mathbb{R}^2$
for the usual Euclidean metric $g_E$,
and they have length $\alpha$ with respect to the metric $g$,
where $g$ is the metric given in (\ref{metric}). 
To avoid confusion, we will denote with a subscript $g$ 
for the geometric
quantities and unit vectors taken with respect to the metric $g$,
while we will use a subscript $E$ when we refer to the Euclidean metric $g_E$.

Given a curve $\gamma(u)=(r(u),x(u))$ in $\mathbb{R}^2_+$,
the speed, unit tangent, and normal vectors
are given by
\begin{equation}\label{1.1}
v=\alpha\sqrt{(x')^2+(r')^2},~~
\mathbf{t}_g=\frac{1}{v}\left(r'\frac{\partial}{\partial r}+x'\frac{\partial}{\partial x}\right),~~
\mathbf{n}_g=\frac{1}{v}\left(-x'\frac{\partial}{\partial r}+r'\frac{\partial}{\partial x}\right).
\end{equation}
We denote $ds=vdu$, where $s$ is the arclength of $\gamma$ with respect
to the metric $g$. One can compute the geodesic curvature: 
\begin{equation}\label{1.2}
\begin{split}
   k_g &
=\frac{1}{v}\left[\frac{x'r''-x''r'}{(x')^2+(r')^2}-\left(\frac{\lambda-1}{r}-\frac{r}{2}\right)x'-\frac{1}{2}xr'\right]\\
&
=\frac{1}{v}\left[k_E-\left(\frac{\lambda-1}{r}-\frac{r}{2}\right)x'-\frac{1}{2}xr'\right].
\end{split}
\end{equation}
Therefore, the geodesic equation for $(\mathbb{R}^2_+,g)$ is given by
\begin{equation}\label{1.3}
\frac{x'r''-x''r'}{(x')^2+(r')^2}=\left(\frac{\lambda-1}{r}-\frac{r}{2}\right)x'+\frac{1}{2}xr'.
\end{equation}
Then the length of the curve $\gamma(u)$, $a\leq u\leq b$, in $(\mathbb{R}^2_+,g)$ is
\begin{equation}\label{1.4}
L_g(\gamma)=\int_a^bvdu.
\end{equation}

\begin{defn}
\emph{We use the notation
$$r_\lambda=\sqrt{2(\lambda-1)}.$$
The self-shrinking cylinder $\mathcal{C}$ corresponds to the geodesic
$(r(u),x(u))=(r_\lambda,u)$. The self-shrinking
half-line $\mathcal{P}$ corresponds to the geodesic
$(r(u),x(u))=(u,0)$.}
\end{defn}

Note that the lengths of $\mathcal{C}$ and $\mathcal{P}$ are respectively given by
\begin{equation}\label{1.5}
\begin{split}
L_g(\mathcal{C})&=\int_{-\infty}^\infty (2(\lambda-1))^{\frac{\lambda-1}{2}}e^{-\frac{2(\lambda-1)+u^2}{4}}du
=2\sqrt{\pi}\left(\frac{2(\lambda-1)}{e}\right)^{\frac{\lambda-1}{2}},\\
L_g(\mathcal{P})&=\int_0^\infty u^{\lambda-1} e^{-\frac{u^2}{4}}du.
\end{split}
\end{equation}

We have the following: 
\begin{prop}
For any $\lambda>1$, there holds
$$L_g(\mathcal{P})<L_g(\mathcal{C}).$$
\end{prop}
\begin{proof}
By change of variables $v=u^2/4$, we can write
\begin{equation}\label{1.10}
L_g(\mathcal{P})=2^{\lambda-1}\int_0^\infty v^{\frac{\lambda-2}{2}} e^{-v}dv=2^{\lambda-1}\Gamma\left(\frac{\lambda}{2}\right).
\end{equation}
It follows from \cite[Theorem 1.5]{Batir} that
$$\Gamma(x+1)\leq\sqrt{2\pi}\left(\frac{x+\frac{1}{2}}{e}\right)^{x+\frac{1}{2}}$$
for any real positive number $x$.
Applying this with $x=\frac{\lambda}{2}-1$, we can estimate the right hand side of (\ref{1.10}) as follows:
\begin{equation*}
\begin{split}
L_g(\mathcal{P})=2^{\lambda-1}\Gamma\left(\frac{\lambda}{2}\right)
\leq 2^{\lambda-1}\sqrt{2\pi}\left(\frac{\lambda-1}{2e}\right)^{\frac{\lambda-1}{2}}
&=\sqrt{2\pi}\left(\frac{2(\lambda-1)}{e}\right)^{\frac{\lambda-1}{2}}\\
&<
2\sqrt{\pi}\left(\frac{2(\lambda-1)}{e}\right)^{\frac{\lambda-1}{2}}=L_g(\mathcal{C}),
\end{split}
\end{equation*}
which proved the assertion.
\end{proof}

The Gauss curvature $K_g$ of $(\mathbb{R}^2_+,g)$ is
\begin{equation}\label{1.6}
K_g=\alpha^{-2}\left(1+\frac{\lambda-1}{r^2}\right).
\end{equation}

\begin{defn}
\emph{The enclosed Gauss area of a simple closed curve $\gamma:\mathbb{S}^1\to\mathbb{R}^2_+$ is
\begin{equation}\label{1.7}
GA_g(\gamma):=\iint_{\Omega}\left(1+\frac{\lambda-1}{r^2}\right)dxdr,
\end{equation}
where $\Omega$ is the region enclosed by $\gamma$.}
\end{defn}

For a $C^1$-curve $\gamma$, it follows from (\ref{1.6}) and (\ref{1.7}) that
the Gauss-Bonnet formula is given by
\begin{equation}\label{1.8}
\iint_{\Omega}\left(1+\frac{\lambda-1}{r^2}\right)dxdr=2\pi-\oint_\gamma k_gds.
\end{equation}

\section{A modified curve shortening flow}

Given a closed curve $\mathbb{R}^2_+$,
we consider the flow that evolves according
to the normal velocity
\begin{equation}\label{1.11}
V_g=\frac{k_g}{K_g},
\end{equation}
where $k_g$ is the geodesic curvature and $K_g$ is the Gauss curvature
at the given point on the curve.
When $K_g$ is positive and uniformly bounded away from zero,
this flow exhibits properties similar to the usual curve shortening flow.

The flow was first studied in \cite{Gage} by Gage. 
Using this flow, Gage proved the existence
of geodesics on spheres.
The idea is similar here, but our manifold $(\mathbb{R}^2_+,g)$ is not closed and
its Gauss curvature is not bounded.
We will use lines $r=r(t)$ as barriers to show that
closed curves in the interior of $\mathbb{R}^2_+$
do not reach regions
where the Gauss curvature blow up in finite time.

\subsection{Evolution of lines $r=C$}

From (\ref{1.2}), we see that the line $r=r_0$ has geodesic curvature:
$$k_g=\frac{1}{\alpha}\left(\frac{r_0}{2}-\frac{\lambda-1}{r_0}\right).$$
We can compute the Euclidean speed $V_E$ in the direction $\displaystyle\frac{\partial}{\partial x}$ as follows: 
$$V_E=\frac{V_g}{\alpha}=\frac{k_g}{\alpha K_g}=\frac{r_0(r_0^2-2(\lambda-1))}{2(r_0^2+\lambda-1)}.$$
Hence, lines on the left hand side of the cylinder $\mathcal{C}$, i.e., $r=r_0<r_\lambda$,
move further to the left.
Similarly, lines on the right hand side of the cylinder $\mathcal{C}$, i.e., $r=r_0>r_\lambda$,
move further to the right.
When $r_0<r_\lambda$,
we have
$V_E\geq -cr$ with a positive constant depending only on $r_0$ and $\lambda$. Therefore,
$$r(t)\geq r_0e^{-ct},~~r_0<r_\lambda.$$
Therefore, no vertical line reaches the $x$-axis in finite time. Similarly,
when $r_0>r_\lambda$, we have
$r(t)\leq r_0e^{\frac{t}{2}}$, and no vertical line goes to infinity in finite time.

\subsection{Short-time and long time existence}

In the case where the positive Gauss curvature is uniformly bounded from above and below away from zero,
the short-time existence of the flow
is obtain by Gage \cite{Gage} or by Angenent \cite{Angenent1}.
Given an initial embedded closed curve
$\gamma_0:\mathbb{S}^1\to \mathbb{R}^2_+$,
we choose $r_0<r_\lambda$ and $r_1>r_\lambda$
such that $\gamma_0(\mathbb{S}^1)$ lies within the slab
$\{r\in(r_0,r_1)\}$.
For $t<1$, the Gauss curvature is then uniformly bounded in $\{r\in (r_0^{-c},r_1e^{1/2}\}$
and we can apply the short-time existence results there.

We have the following proposition;
its proof is the same as that of  \cite[Proposition 5]{Drugan&Nguyen}
and is omitted.

\begin{prop}\label{prop5}
Denote $\gamma_t$ the evolution of $\gamma_0$ at time $t$.\\
(i) If $\gamma_0$ is an embedded curve, so is $\gamma_t$.\\
(ii) If $\gamma_0$ is symmetric with respect to the reflections across the $r$-axis and the image of $\gamma_0$
in the first quadrant is a graph over the $r$-axis, the same properties
(symmetric and graphical) hold for all $\gamma_t$ as long as the flow exists.
\end{prop}

Along the flow, the arclength evolves according to
$$\frac{\partial}{\partial t}ds
=-k_g V_gds=-\frac{k_g^2}{K_g}ds.$$
This implies that the length $L_g(\gamma_t)$ evolves by
\begin{equation}\label{1.9}
\frac{d}{dt}L_g(\gamma_t)
=-\int_{\gamma_t}\frac{k_g^2}{K_g}ds.
\end{equation}
If the $\gamma_t$'s
are simple and closed, then they are boundaries of domains $\Omega_t$.
Using the Gauss-Bonnet formula in (\ref{1.8}), we obtain
\begin{equation*}
\frac{d}{dt}\iint_{\Omega_t}K_g dA_g
=-\oint_{\gamma_t}K_g V_gds=-\oint_{\gamma_t}k_gds=-2\pi+\iint_{\Omega_t}K_g dA_g.
\end{equation*}
Here, $V_g$ and $k_g$ are the normal velocity and geodesic curvature in the direction of the inward normal to $\Omega$.
Hence, we have the following:

\begin{prop}\label{prop6}
If the Gauss area enclosed by the initial curve is equal to $2\pi$,
then the Gauss area enclosed by $\gamma_t$ is also $2\pi$ whenever the flow exists.
\end{prop}

We have the following proposition; its proof is the same as that of  \cite[Proposition 5]{Drugan&Nguyen}
and is omitted.

\begin{prop}\label{prop7}
Let $\gamma_0$ be a simple closed curve.
If the domain enclosed by $\gamma_0$ has Gauss area being equal to $2\pi$,
then the evolution of $\gamma_0$ with normal velocity $V_g=k_g/K_g$ exists for all time.
\end{prop}

\section{A family of initial curves}

We consider rectangles $R[a,b,c]$ with vertices $(a,-c)$, $(a,c)$,
$(b,c)$, $(b,-c)$, with $a<b$ and $c>0$.
It follows from Proposition \ref{prop7} that if our initial rectangle encloses
Gauss area equal to $2\pi$,
the flow will exist for all time.
If in addition, its perimeter
is less than
$2L_g(\mathcal{P})$,
the flow cannot converges
to a double cover of a plane,
because it decreases length according to (\ref{1.9}).
There are infinitely many rectangles satisfying
both of these conditions.
For example,
when $a$ and $b$ are large,
the Gauss area
$\displaystyle\iint\left(1+\frac{\lambda-1}{r^2}\right)dxdr\sim \iint dxdr$,
but the perimeter will be small because of the exponential weight on the metric.
When $a$ and $b$ close to $0$,
the integral $\displaystyle\int_0^1\frac{1}{r^2}dr$ is unbounded.
Therefore, it is possible
to have a tiny rectangle with Gauss area $2\pi$ and with very small perimeter.
Here, we choose a continuous one parameter family that bridges
these two extremes.

\begin{defn}
\emph{We denote by $L(a,b,c)$ the perimeter of the rectangle $R[a,b,c]$ with respect
to the metric $g$, i.e.,}
\begin{equation}\label{2.1}
L_g[a,b,c]
:=2(a^{\lambda-1} e^{-a^2/4}+b^{\lambda-1} e^{-b^2/4})\int_0^ce^{-x^2/4}dx
+2e^{-c^2/4}\int_a^b r^{\lambda-1} e^{-r^2/4}dr.
\end{equation}
\end{defn}

Fix $c_0$ to be the positive real number such that
\begin{equation}\label{2.2}
\frac{e^{-c_0^2/4}}{\int_0^{c_0}e^{-x^2/4}dx}=2.
\end{equation}
The number $c_0$ exists and is unique, and is approximately $0.481$.
Also, let $\lambda_0$ be the positive real number such that
$$\sqrt{2\lambda_0+1-\sqrt{8\lambda_0+1}}=\sqrt{2(\lambda_0+4)}-2,$$
which is equal to $(28+11\sqrt{7})/9$, which is approximately $6.3448$.

\begin{prop}\label{prop9}
Suppose that $\lambda>1$.
For every $a,b\in (0,\infty)$, we have
$$L_g[a,b,c]<2L_g(\mathcal{P}).$$
\end{prop}
\begin{proof}
In view of (\ref{2.2}), we can write (\ref{2.1}) as
$$L_g(a,b,c_0)=f_\lambda(a)+g_\lambda(b),$$
where $M=\displaystyle\int_0^{c_0}e^{-x^2/4}dx$ and
\begin{equation*}
\begin{split}
f_\lambda(a)&=2a^{\lambda-1} e^{-a^2/4}M+4M\int_a^0r^{\lambda-1} e^{-r^2/4}dr,\\
g_\lambda(b)&=2b^{\lambda-1} e^{-b^2/4}M+4M\int_0^br^{\lambda-1} e^{-r^2/4}dr.
\end{split}
\end{equation*}
The derivative of $f$ is
$$f_\lambda'(a)=-a^{\lambda-2}e^{-a^2/4}M\big(a^2+4a-2(\lambda-1)\big).$$
Therefore, $\displaystyle\max_{a\in[0,\infty)}f_\lambda(a)$
is achieved at $a_\lambda=-2+\sqrt{2(1+\lambda)}$.
Similarly, one can show that
$\displaystyle\max_{b\in[0,\infty)}g_\lambda(b)$
is achieved at $b_\lambda=2+\sqrt{2(1+\lambda)}$.
It follows from (\ref{1.5}) that
$$L_g(\mathcal{P})=\int_0^\infty u^{\lambda-1} e^{-\frac{u^2}{4}}du.$$
Hence, it suffices to show that
\begin{equation}\label{2.4}
f_\lambda(a_{\lambda})+g_\lambda(b_\lambda)<2\int_0^\infty r^{\lambda-1} e^{-\frac{r^2}{4}}dr
\end{equation}
for any $\lambda>1$.

To prove (\ref{2.4}), we are going to prove the following two lemmas:

\begin{lem}\label{induction_lem}
The inequality \eqref{2.4} holds whenever $1\leq \lambda\leq 7$.
\end{lem}

\begin{lem}\label{lem10} There holds
\begin{equation}\label{2.3}
f_{\lambda+2}(a_{\lambda+2})+g_{\lambda+2}(b_{\lambda+2})
<2\lambda\big(f_{\lambda}(a_{\lambda})+g_{\lambda}(b_{\lambda})\big)
\end{equation}
for all $\lambda>\lambda_0$.
\end{lem}

Let us assume Lemma \ref{induction_lem} and Lemma \ref{lem10} for the time being;
and we can see that (\ref{2.4}) holds for all $\lambda>1$.
It follows from integration by parts that
\begin{equation*}
\begin{split}
\int_0^\infty r^{\lambda+1} e^{-\frac{r^2}{4}}dr
&=\Big[-2r^{\lambda}e^{-\frac{r^2}{4}}\Big]_0^\infty+2\lambda\int_0^\infty r^{\lambda-1} e^{-\frac{r^2}{4}}dr\\
&=2\lambda\int_0^\infty r^{\lambda-1} e^{-\frac{r^2}{4}}dr.
\end{split}
\end{equation*}
This together with Lemma \ref{lem10} implies that
if (\ref{2.4}) holds for a particular value of $\lambda>\lambda_0$, then
(\ref{2.4}) holds for $\lambda+2$.
Since Lemma \ref{induction_lem} asserts that (\ref{2.4}) holds for all $1\leq \lambda\leq 7$,
we can conclude that  (\ref{2.4}) holds for all $\lambda>1$.
This proves Proposition \ref{prop9} by assuming Lemma \ref{induction_lem} and Lemma \ref{lem10}.
\end{proof}

It remains to prove Lemma \ref{induction_lem} and Lemma \ref{lem10}.

\begin{proof}[Proof of Lemma \ref{induction_lem}]
Note that left hand side of (\ref{2.4}) can be written as 
\begin{equation}\label{2.19}
2M\left(a_\lambda^{\lambda-1} e^{-a_\lambda^2/4}+b_\lambda^{\lambda-1} e^{-b_\lambda^2/4}\right)
+4M\int_{a_\lambda}^{b_\lambda}r^{\lambda-1} e^{-r^2/4}dr.
\end{equation}
It follows from (\ref{2.2}) that 
$$M=\int_0^{c_0}e^{-x^2/4}dx=\frac{1}{2}e^{-c_0^2/4}<\frac{1}{2}.$$
This together with (\ref{2.19}) implies that the left hand side of (\ref{2.4})
is less than or equal to 
$$
a_\lambda^{\lambda-1} e^{-a_\lambda^2/4}+b_\lambda^{\lambda-1} e^{-b_\lambda^2/4}
+2\int_{a_\lambda}^{b_\lambda}r^{\lambda-1} e^{-r^2/4}dr.
$$
This implies that, in order to prove (\ref{2.4}), it suffices to prove 
\begin{equation}\label{2.20}
a_\lambda^{\lambda-1} e^{-a_\lambda^2/4}+b_\lambda^{\lambda-1} e^{-b_\lambda^2/4}
<2\left(\int_0^{a_\lambda}+\int_{b_\lambda}^\infty\right)r^{\lambda-1} e^{-r^2/4}dr.
\end{equation}
By the change of variables, we can further rewrite (\ref{2.20}) as 
$$
a_\lambda^{\lambda-1} e^{-a_\lambda^2/4}+b_\lambda^{\lambda-1} e^{-b_\lambda^2/4}
<2^\lambda\left(\int_0^{a_\lambda^2/4}+\int_{b_\lambda^2/4}^\infty\right)x^{\frac{\lambda}{2}-1} e^{-x}dx, 
$$
or equivalently, 
\begin{equation}\label{2.21}
\left(\frac{a_\lambda^2}{4}\right)^{\frac{\lambda-1}{2}} e^{-a_\lambda^2/4}+\left(\frac{b_\lambda^2}{4}\right)^{\frac{\lambda-1}{2}} e^{-b_\lambda^2/4}
<2\left(\int_0^{a_\lambda^2/4}+\int_{b_\lambda^2/4}^\infty\right)x^{\frac{\lambda}{2}-1} e^{-x}dx, 
\end{equation}
Note that the right hand side of (\ref{2.21}) can be written as 
\begin{equation}\label{2.22}
2\left(\Gamma\Big(\frac{\lambda}{2}\Big)-\Gamma\Big(\frac{\lambda}{2},\frac{a_\lambda^2}{4}\Big)+\Gamma\Big(\frac{\lambda}{2},\frac{b_\lambda^2}{4}\Big)\right),
\end{equation}
where
$$\Gamma(u,v)=\int_{v}^\infty x^{u-1}e^{-x}dx$$
is the incomplete gamma function. 
Using Mathematica, one can see that the function in (\ref{2.22}) 
is greater than $1$ whenever $1\leq\lambda\leq 7$. 
On the other hand, using Mathematica again, 
one can see that the function on the left hand side of (\ref{2.21}) 
is less than $1$ whenever $1\leq\lambda\leq 7$. 
This proves Lemma \ref{induction_lem}.   
\end{proof}

\begin{proof}[Proof of Lemma \ref{lem10}]
It follows from the integration by parts that
\begin{equation*}
2\lambda\int_a^b r^{\lambda-1} e^{-\frac{r^2}{4}}dr
=2b^{\lambda}e^{-\frac{b^2}{4}}-2a^{\lambda}e^{-\frac{a^2}{4}}
+\int_a^b r^{\lambda+1} e^{-\frac{r^2}{4}}dr.
\end{equation*}
Therefore, for all $a, b>0$, we have
\begin{equation}\label{2.5}
\begin{split}
2\lambda f_\lambda(a)-2f'_{\lambda+1}(a)-f_{\lambda+2}(a)&=0,\\
2\lambda g_\lambda(b)-2g'_{\lambda+1}(b)-g_{\lambda+2}(b)&=0.
\end{split}
\end{equation}
Since $f'_{\lambda+1}(a_{\lambda+1})=g'_{\lambda+1}(b_{\lambda+1})=0$,
the inequality (\ref{2.3})
is equivalent to
\begin{equation}\label{2.6}
\int_{a_{\lambda+1}}^{a_{\lambda+2}}
f'_{\lambda+2}(u)du+\int_{b_{\lambda+1}}^{b_{\lambda+2}}
g'_{\lambda+2}(v)dv
<2\lambda\left(\int_{a_{\lambda+1}}^{a_{\lambda}}
f'_{\lambda}(u)du+\int_{b_{\lambda+1}}^{b_{\lambda}}
g'_{\lambda}(v)dv\right).
\end{equation}
Using the variables $u=s-2$ and $v=s+2$, we obtain the following equivalent form:
\begin{equation*}
\begin{split}
&\int_{r_{\lambda+3}}^{r_{\lambda+4}}
\big(2(\lambda+3)-s^2\big)\left[(s-2)^{\lambda}e^{-\frac{(s-2)^2}{4}}+(s+2)^{\lambda}e^{-\frac{(s+2)^2}{4}}\right]ds\\
&<2(\lambda+1) \int_{r_{\lambda+2}}^{r_{\lambda+3}}
\big(y^2-2(\lambda+1)\big)\left[(y-2)^{\lambda-2}e^{-\frac{(y-2)^2}{4}}+(y+2)^{\lambda-2}e^{-\frac{(y+2)^2}{4}}\right]dy.
\end{split}
\end{equation*}
For $\lambda>0$, we define the function
\begin{equation*}
h_\lambda(s)=e^{-\frac{s^2}{4}}s^{\lambda},~~H_\lambda(s)=h_\lambda(s-2)+h_\lambda(s+2).
\end{equation*}
Therefore, if we substitute $s=r_{\lambda+4}-t$ and $y=r_{\lambda+2}+t$, we get another equivalent form:
\begin{equation}\label{2.7}
\begin{split}
&\int_0^{r_{\lambda+4}-r_{\lambda+3}}t(2r_{\lambda+4}-t)H_\lambda(r_{\lambda+4}-t)dt\\
&\hspace{4mm}<\int_0^{r_{\lambda+3}-r_{\lambda+2}}2\lambda\, t(2r_{\lambda+2}+t)H_{\lambda-2}(r_{\lambda+2}+t)dt.
\end{split}
\end{equation}
Because of concavity, we have $r_{\lambda+4}-r_{\lambda+3}<r_{\lambda+3}-r_{\lambda+2}$.
Therefore, in order to prove (\ref{2.3})
and equivalently (\ref{2.7}),
it suffices to prove the following
pointwise inequality:
\begin{equation}\label{2.8}
(2r_{\lambda+4}-t)H_\lambda(r_{\lambda+4}-t)<2 \lambda (2r_{\lambda+2}+t)H_{\lambda-2}(r_{\lambda+2}+t)
\end{equation}
for all $0<t<r_{\lambda+4}-r_{\lambda+3}$.

We claim that the inequality (\ref{2.8})
is at its tightest when $t=0$.
More precisely, we are going to show that
\begin{equation}\label{2.10}
\begin{split}
(2r_{\lambda+4}-t)H_\lambda(r_{\lambda+4}-t)&\leq 2r_{\lambda+4}H_\lambda(r_{\lambda+4}),\\
4 \lambda\, r_{\lambda+2}H_{\lambda-2}(r_{\lambda+2})&\leq 2\lambda(2r_{\lambda+2}+t)H_{\lambda-2}(r_{\lambda+2}+t)
\end{split}
\end{equation}
for all $0<t<r_{\lambda+4}-r_{\lambda+3}$.
We compute
\begin{equation}\label{2.9}
\begin{split}
h'_\lambda(s)&=e^{-\frac{s^2}{4}}s^{\lambda-1}\left(\lambda-\frac{s^2}{2}\right),\\
h''_\lambda(s)&=e^{-\frac{s^2}{4}}s^{\lambda-2}\left(\lambda(\lambda-1)-\frac{2\lambda+1}{2}s+\frac{s^2}{4}\right).
\end{split}
\end{equation}
It follows from quadratic formula that
$h''(s)\geq 0$ if and only if
$s^2\leq s_0^2=2\lambda+1-\sqrt{8\lambda+1}$
or $s^2\geq s_1^2=2\lambda+1+\sqrt{8\lambda+1}$.
For $H_\lambda''$, we are concerned about
$s\in (r_{\lambda+3}-2, r_{\lambda+5}-2)\cup(r_{\lambda+3}+2, r_{\lambda+5}+2)$
and we see that
$$r_{\lambda+5}-2<s_0~~\mbox{ and }~~s_1<r_{\lambda+3}+2,$$
since
$$s_0=\sqrt{2\lambda+1-\sqrt{8\lambda+1}}>\sqrt{2(\lambda+4)}-2=r_{\lambda+5}-2~~\mbox{ whenever }\lambda>\lambda_0,$$
and
$$s_1=\sqrt{2\lambda+1+\sqrt{8\lambda+1}}<\sqrt{2(\lambda+2)}+2=r_{\lambda+3}+2~~\mbox{ whenever }\lambda>1.$$
Therefore, $H_\lambda''(s)>0$ for $s\in (r_{\lambda+3},r_{\lambda+4})$ and $\lambda>\lambda_0$.

We will now show that $H_\lambda'(r_{\lambda+3})>0$, or equivalently,
\begin{equation*}
Q_\lambda:=-\frac{h_\lambda'(r_{\lambda+3}+2)}{h_\lambda'(r_{\lambda+3}-2)}<1.
\end{equation*}
By (\ref{2.9}), we compute
\begin{equation}\label{2.11}
Q_\lambda=e^{-2r_{\lambda+3}}\left(\frac{r_{\lambda+3}+2}{r_{\lambda+3}-2}\right)^{\lambda}
=e^{-2r_{\lambda+3}}\left(1+\frac{4}{r_{\lambda+3}-2}\right)^{(r_{\lambda+3}^2-4)/2}.
\end{equation}

One has the following: (see \cite[Appendix]{Drugan&Nguyen})
\begin{equation}\label{2.12}
\begin{split}
f(x)=\left(1+\frac{a}{x}\right)^{x^2}\nearrow e^{ax}e^{-\frac{a^2}{2}},~~a>0,\\
g(x)=\left(1-\frac{a}{x}\right)^{x^2}\searrow e^{-ax}e^{-\frac{a^2}{2}},~~a>0
\end{split}
\end{equation}
as $x\to\infty$. Note that $r_{\lambda+3}>2$ whenever $\lambda>1$.
Taking $x=\displaystyle\frac{r_{\lambda+3}-2}{\sqrt{2}}>0$
and $a=2\sqrt{2}$, we can estimate
\begin{equation}\label{2.13}
\left(1+\frac{4}{r_{\lambda+3}-2}\right)^{\frac{(r_{\lambda+3}-2)^2}{2}}
\leq e^{-2(r_{\lambda+3}-2)}e^{-4}.
\end{equation}
Since
$$\left(1+\frac{1}{x}\right)^x\nearrow e^x$$
as $x\to\infty$, we can estimate
\begin{equation}\label{2.14}
\left(1+\frac{4}{r_{\lambda+3}-2}\right)^{2(r_{\lambda+3}-2)}
\leq e^{2(r_{\lambda+3}-2)}.
\end{equation}
Using (\ref{2.13}) and (\ref{2.14}),
we can estimate the right hand side of (\ref{2.11})
as follows:
\begin{equation*}
Q_\lambda=e^{-2r_{\lambda+3}}
\left(1+\frac{4}{r_{\lambda+3}-2}\right)^{\frac{(r_{\lambda+3}-2)^2}{2}+2(r_{\lambda+3}-2)}
\leq e^{-2r_{\lambda+3} -4}<1,
\end{equation*}
since $2r_{\lambda+3}+4\geq 0$.

The rest of the proof is dedicated to proving
the following equivalent formulation of the inequality
(\ref{2.8}) at $t=0$:
\begin{equation}\label{2.15}
\ln(\mathcal{H}(r_{\lambda+4}))<\ln\left(e\frac{\lambda(\lambda+3)}{(\lambda+1)^2}\left(\frac{\lambda+1}{\lambda+3}\right)^{\frac{\lambda+3}{2}}\mathcal{H}(r_{\lambda+2})\right)=:I+\ln(\mathcal{H}(r_{\lambda+2})),
\end{equation}
where $m=m(s)=\frac{s^2}{2}-3$ and $\mathcal{H}(s):=s^{-m}H_m(s)e^{(s^2+4)/4}
=(1-\frac{2}{s})^{m}e^s+(1+\frac{2}{s})^{m}e^{-s}$.
We closely follow the proof of \cite[Proposition 9]{Drugan&Nguyen}.

We recall the power expansion of $\ln$ for $0<y<1$:
\begin{equation}\label{2.16}
\ln(1+y)=y-\frac{y^2}{2}+\cdots+(-1)^{j+1}\frac{y^j}{j}+R_j^+(y),~~-\frac{y^{2K+2}}{2K+2}\leq R_{2K+1}^+(y)\leq 0,
\end{equation}
and
\begin{equation}\label{2.17}
\ln(1-y)=-y-\frac{y^2}{2}-\cdots-\frac{y^j}{j}+R_j^-(y),~~-\frac{y^{j+1}}{(j+1)(1-y)}\leq R_j^-(y)\leq 0.
\end{equation}
Separating the even and odd powers, we have
\begin{equation*}
\begin{split}
\left(\frac{s^2}{2}-3\right)\ln\left(1-\frac{2}{s}\right)+s&=E_{2K-2}(s)+O_{2K-1}(s)+\mathcal{R}^-_{2K-1}(s),\\
\left(\frac{s^2}{2}-3\right)\ln\left(1+\frac{2}{s}\right)+s&=E_{2K-2}(s)-O_{2K-1}(s)+\mathcal{R}^+_{2K-1}(s).
\end{split}
\end{equation*}
For the odd powers, we just remark that the function $O_{2K-1}$ is decreasing in $s$.
We need the formulas for the even powers and the reminder
\begin{equation*}
\begin{split}
E_{2K-2}(s)&:=-1+\sum_{j=1}^{K-1}\left(\frac{3}{2j}-\frac{2}{2j+2}\right)\left(\frac{2}{s}\right)^{2j},\\
\mathcal{R}^{\pm}_{2K-1}(s)&:=-3R_{2K-1}^{\pm}(\frac{2}{s})+\frac{s^2}{2}R_{2K+1}^{\pm}(\frac{2}{s}).
\end{split}
\end{equation*}
From (\ref{2.16}) and (\ref{2.17}), we note that the (lower) bound for $R_j^-$ dominates,
therefore it suffices to consider $\mathcal{R}^-$. We obtain
\begin{equation*}
 -\frac{2}{(K+1)(s-2)}\left(\frac{2}{s}\right)^{2K-1}
 \leq  \mathcal{R}^-_{2K-1}(s)
 \leq
 \frac{3}{K(s-2)}\left(\frac{2}{s}\right)^{2K-1}.
\end{equation*}
Therefore, for $m(s)=\frac{s^2}{2}-3$, we have
\begin{equation*}
\begin{split}
\mathcal{H}(s)&>2\exp\left(E_{2K-2}(s)-\frac{2}{(K+1)(s-2)}\left(\frac{2}{s}\right)^{2K-1} \right)\cosh(O_{2K-1}(s)),\\
\mathcal{H}(s)&<2\exp\left(E_{2K-2}(s)+\frac{3}{K(s-2)}\left(\frac{2}{s}\right)^{2K-1} \right)\cosh(O_{2K-1}(s)).
\end{split}
\end{equation*}
Because $\cosh(O_{2K-1}(s))>\cosh(O_{2K+1}(s+\epsilon))$,
to finish proving (\ref{2.15}), we show that
\begin{equation}\label{2.18}
\left(E_{2K-2}(s)-E_{2K-2}(s+\epsilon)-\left(\frac{2}{K+1}+\frac{3}{K}\right)\frac{1}{s-2}\left(\frac{2}{s}\right)^{2K-1}\right)+I>0
\end{equation}
for $s=r_{\lambda+2}$ and $s+\epsilon=r_{\lambda+4}$
and where $I$ was defined implicitly in (\ref{2.15}).
The values of $s$ and $s+\epsilon$ are now fixed.

First, we estimate the last term of (\ref{2.18}) by expanding in power series and keeping lower order terms
in $(\lambda+1)^{-1}$ and $(\lambda+3)^{-1}$:
\begin{equation*}
\begin{split}
I &>1+\left(\frac{1}{\lambda+1}-\frac{2}{(\lambda+1)^2}\right)-\left(\frac{1}{\lambda+1}-\frac{2}{(\lambda+1)^2}\right)^2\\
&-1-\frac{1}{\lambda+3}-\frac{1}{3}\left(\frac{2}{\lambda+3}\right)^2-\frac{1}{4}\left(\frac{2}{\lambda+3}\right)^3\frac{\lambda+3}{\lambda+1}\\
&>\frac{2}{(\lambda+1)(\lambda+3)}-\frac{3}{(\lambda+1)^2}-\frac{4}{3(\lambda+3)^2}\\
&>-\frac{3}{(\lambda+1)(\lambda+3)}-\frac{15}{(\lambda+1)^2(\lambda+3)^2}.
\end{split}
\end{equation*}
The terms of lower order in $E_{2K-2}(s)-E_{2K-2}(s+\epsilon)$ are explicitly given by
\begin{equation*}
\begin{split}
E_{2K-2}(s)-E_{2K-2}(s+\epsilon)&
>\frac{4}{2(\lambda+1)}-\frac{4}{2(\lambda+3)}+\frac{20}{3(2(\lambda+1))^2}-\frac{20}{3(2(\lambda+3))^2}\\
&=\frac{4}{(\lambda+1)(\lambda+3)}+\frac{20}{3(\lambda+1)^2(\lambda+3)^2}.
\end{split}
\end{equation*}
Because $\frac{4}{s}>\frac{1}{s-2}$ for $s>8/3$ (i.e. when $\lambda>23/9$), the error term can be estimated by
$$-\left(\frac{2}{K+1}+\frac{3}{K}\right)\frac{1}{s-2}\left(\frac{2}{s}\right)^{2K-1}
>-10K^{-1}(2/s)^{2K}.$$
Recalling that $s=r_{\lambda+2}$, we obtain
$$E_{2K-2}(s)-E_{2K-2}(s+\epsilon)+I-10\frac{2^{2K}}{Ks^{2K}}>\frac{1}{(\lambda+1)(\lambda+3)}-
10\frac{2^K}{K(\lambda+1)^K},$$
which is positive for $K$ large enough. This finishes the proof of Lemma  \ref{lem10}.
\end{proof}

For a fixed $\lambda$,
we round the corners of the rectangle while keeping the length
less than $2L_g(\mathcal{P})$.

\begin{defn}\label{defn1}
\emph{Let us define $\mathcal{R}[a,b,c_0]$
as the rectangle $R[a,b,c_0]$ with rounded corners.
Given $\lambda>1$, all the rectangles are rounded off in the same manner
such that
$L_g(\mathcal{R}[a,b,c_0])<2L_g(\mathcal{P})$ for all $a,b>0$.
}
\end{defn}

As an immediate corollary, we can extract a one parameter family of
rounded rectangle with enclosed Gauss area equal to $2\pi$.

\begin{cor}
There is a smooth function $\varphi: \mathbb{R}_+\to\mathbb{R}_+$
with $\varphi(a)>a$ such that the family of rectangles
$\mathcal{R}[a,\varphi(a),c_0]$ satisfies
$$GA_g(\mathcal{R}[a,\varphi(a),c_0])=2\pi~~\mbox{ and }~~
L_g(a,\varphi(a),c_0)<2L_g(\mathcal{P}).$$
\end{cor}

Note that $\displaystyle\lim_{a\to 0}\varphi(a)=0$.

\begin{prop}
  Let $\Phi:\mathbb{R}_+\times\mathbb{R}_+\to C^0(\mathbb{S}^2,\mathbb{R}^2_+)$
be the map with the following properties:\\
(i) $\Phi(a,0)=\mathcal{R}[a,\varphi(a),c_0]$.\\
(ii) For fixed $a$, $\Phi(a,t)$ satisfies the evolution equation \eqref{1.11}.\\
Then there exists $a_0\in\mathbb{R}_+$ such that $\Phi(a_0,t)$ intersects the cylinder $\mathcal{C}$
for all time $t\in\mathbb{R}_+$.
\end{prop}
\begin{proof}
The set of curves that do not intersect the cylinder $\mathcal{C}$ is split into two disjoint sets:
\begin{equation*}
\begin{split}
   A_1&=\Big\{\mbox{continuous closed curves }\gamma:\mathbb{S}^1\to\mathbb{R}^2_+\,|\,\gamma(s)<r_\lambda, s\in\mathbb{S}^1\Big\},\\
   A_2&=\Big\{\mbox{continuous closed curves }\gamma:\mathbb{S}^1\to\mathbb{R}^2_+\,|\,\gamma(s)>r_\lambda, s\in\mathbb{S}^1\Big\}.
\end{split}
\end{equation*}
The cylinder $\mathcal{C}$ is a self-shrinker, and as a result, it is stationary under the flow.
By the maximum principle,
if $\Phi(a,t_0)\in A_i$ for some $i=1,2$ and some $t_0$,
then $\Phi(a,t)\in A_i$ for all $t\geq t_0$.
Now we consider the following subsets of $\mathbb{R}_+$:
$$U_i=\{a\in\mathbb{R}_+\,|\,\exists t>0\mbox{ with }\Phi(a,t)\in A_i\}.$$
Both $U_1$ and $U_2$ are open and $U_1\cap U_2=\emptyset$.
Since $\mathbb{R}_+$ is connected,
we must have $U_1\cup U_2\neq \mathbb{R}_+$.
This proves our claim.
\end{proof}

\section{Convergence to a geodesic}

We have the following two propositions;
their proofs are almost the same as that of  \cite[Proposition 14]{Drugan&Nguyen}
and  \cite[Proposition 15]{Drugan&Nguyen} with a small modification,
and are omitted.

\begin{prop}\label{prop14}
There is a sequence $t_i$ so that
\begin{equation}\label{3.1}
\int_{\gamma_t\cap E}|k_g|ds\to 0
\end{equation}
for any compact subset $E$ of the open half-space $\mathbb{R}^2_+$.
\end{prop}

\begin{prop}\label{prop15}
Let $t_i$ be the sequence from Proposition \ref{prop14}
(or possibly one of its subsequences) and let $E$ be a compact set in $\mathbb{R}^2_+$.
If for some $x_i\in\mathbb{S}^1$, $i\in\mathbb{N}$,
the sequence $\{\gamma_i(x_i)\}$ converges to a point $P\in E$, then there exists
a subsequence $i_j$ so that the connected component
of $\gamma_{i_j}\cap E$ containing $\gamma_{i_j}(x_{i_j})$ converges in $C^1$ in $E$.
The limit curve contains $P$ and satisfies the geodesic equation in $E$.
\end{prop}

\subsection{Properties of positive solutions to
the graphical geodesic equation}

When $\gamma=(r,f(r))$
is a graph over the $r$-axis, the geodesic equation (\ref{1.3}) is equivalent to
\begin{equation}\label{3.2}
\frac{f''}{(1+f'^2)}=\left(\frac{r}{2}-\frac{\lambda-1}{r}\right)f'-\frac{1}{2}f.
\end{equation}

We have the following lemmas; their proofs are the same as \cite[Lemma 16]{Drugan&Nguyen}, \cite[Lemma 17]{Drugan&Nguyen}
and \cite[Lemma 18]{Drugan&Nguyen}
by changing $r_n$ to $r_\lambda$,
and are omitted.

\begin{lem}\label{lem16}
A positive solution $f:I\to\mathbb{R}$ to \eqref{3.2} has the following properties:\\
(i) the function $f$ does not have a local minimum in $I$,\\
(ii) if $f'(\rho_0)\geq 0$ for some $\rho\in (0,r_\lambda)\cap I$, then $f''<0$ in $(\rho_0,r_\lambda)\cap I$,\\
(iii) if $f'(\rho_0)\leq 0$ for some $\rho\in (r_\lambda,\infty)\cap I$, then $f''<0$ in $(r_\lambda,\rho_0)\cap I$.
\end{lem}

\begin{lem}\label{lem17}
Let $f:I\to\mathbb{R}$ be a positive function with $f''\leq 0$ on the interval $(\xi_1,\xi_3)\subseteq I$ and
such that the Gauss area under its graph is less than $\pi$, i.e.,
$$\int_{\xi_1}^{\xi_3}f(r)\left(1+\frac{\lambda-1}{r^2}\right)dr\leq \pi.$$
Then for any $\xi_2\in[\xi_1,\xi_3]$, we have
$$I(\xi_2)+II(\xi_2)+III(\xi_2)\leq \pi,$$
where
\begin{equation*}
\begin{split}
I(\xi_2)&:=f(\xi_2)(\xi_3-\xi_1)/2,\\
II(\xi_2)&:=\lambda\, f(\xi_2)\lim_{r\to\xi_2}\big(\ln(r/\xi_1)+\xi_1/r-1\big)/(r-\xi_1),\\
III(\xi_2)&:=\lambda\, f(\xi_2)\lim_{r\to\xi_2}\big(\ln(r/\xi_3)+\xi_3/r-1\big)/(\xi_3-r).
\end{split}
\end{equation*}
\end{lem}

\begin{lem}\label{lem18}
If $f:(a,b)\to\mathbb{R}$
is maximally extended solution to \eqref{3.2},
then $a<r_\lambda<b$.
\end{lem}

With Lemma \ref{lem17},
one can follow the proof of \cite[Lemma 19]{Drugan&Nguyen}
to prove the following:

\begin{lem}\label{lem19}
Let $[\epsilon_0,R_0]\to\mathbb{R}$ be a solution to \eqref{3.2}
with $r_\lambda\in (\epsilon_0,R_0)$.
Suppose in addition that $f$ is positive on $[\epsilon_0,R_0]$
and the Gauss area under the graph $f$ is at most $\pi$.
Then
$$f(r)\leq M_1:=\max\left(f(\epsilon_0),\frac{2\pi}{r_\lambda-\epsilon_0}\right),~~r\in [\epsilon_0,r_\lambda]$$
and
$$f(r)\leq M_2:=\max\left(f(R_0),\frac{2\pi}{R_0-r_\lambda}\right),~~r\in [r_\lambda,R_0].$$
\end{lem}

\begin{defn}\label{defn2}
\emph{Let $L_0$ be the length of our initial rectangle $\mathcal{R}[a_0,\varphi(a_0),c_0]$.
Let $\rho$ be a small constant and $R$ a large constant such that
the graph of any positive function $h$ with domain
$(1/R,R)$ and $h\leq \eta$ has length greater than $\frac{1}{2}L_0$.}
\end{defn}

Note that the constants exist because $L_0<2L_g(\mathcal{P})$
 in view of Definition \ref{defn1}.
Then we have the following proposition;
its proof is the same as that of \cite[Proposition 21]{Drugan&Nguyen}.

 \begin{prop}\label{prop21}
Let $R$ be as in Definition \ref{defn2}.
There exists a constant $\delta>0$ such that
the maximally extended solution $f:(a,b)\to \mathbb{R}$
to the graphical geodesic equation \eqref{3.2} with initial conditions
$$0\leq f(\rho_0)\leq\delta~~\mbox{ and }~~|f'(\rho_0)|\leq\delta,~~\rho_0\in(1/2,2r_\lambda),$$
has a domain containing $[-1/R,R]$ and presents one of the following two behaviors:\\
(i) the graph $f([-1/R,R])$ stays above the $r$-axis and has length greater than $\frac{1}{2}L_0$, or\\
(ii) the graph of $f$ crosses the $r$-axis with finite slope at a point in $(-1/R,R)$.
 \end{prop}

\subsection{Convergence to a Shrinking Doughnut}

Let us denote $a(i)$ and $b(i)$ the points where $\gamma_i$
intersects the $r$-axis,
with the following convention
$$a(i)<r_\lambda<b(i).$$
We can extract convergent subsequences $a(i_j)$ and $b(i_j)$ for which
\begin{equation*}
\begin{split}
\lim_{j\to\infty}a(i_j)=a_\infty,~~a_\infty\in [0,r_\lambda],\\
\lim_{j\to\infty}b(i_j)=a_\infty,~~b_\infty\in [r_\lambda,\infty],
\end{split}
\end{equation*}

\begin{defn}
For each $i\in\mathbb{N}$, let $f_i: (a(i),b(i))\to[0,\infty)$
be the function such that
$$\big\{(r,f_i(r)), r\in(a(i),b(i)) \big\}\subset \gamma_i(\mathbb{S}^1).$$
\end{defn}

Note that the $f_i$'s do not in general satisfy (\ref{3.2}).
However, we can extract a subsequence, also denoted by
$f_i$, that converges in $C^1$ in any compact set to a geodesic $f$ thanks to Proposition \ref{prop15}.

Following \cite[Lemmas 23 and 24]{Drugan&Nguyen}, we can show that
$a_\infty\in (0,r_\lambda)$ and $b_\infty\in (r_\lambda,\infty)$.

Now consider the compact set
$$K=[a_\infty-1,b_\infty+1]\times[-M, M]$$
where $M$ is as in Lemma \ref{lem19}
with $\epsilon=a_\infty$ and $R_0=b_\infty$.
Our subsequence
$\gamma_i$ converges to a geodesic $\gamma_\infty$ in $K$.
Because the intersections
of the $\gamma_i$'s with the $r$-axis
are eventually in $K$,
the intersections of $\gamma_\infty$ with the $r$-axis
are in $K$.
The curve $\gamma_\infty$ is a connected $C^1$ curve that
satisfies the geodesic equation in $K$,
therefore it is smooth.
This proves Theorem \ref{main}.

\bibliographystyle{amsplain}

\end{document}